\documentclass[11pt,twoside]{amsart}
\usepackage[english]{babel}
\usepackage{amsmath, amsthm, amssymb, amsfonts}
\usepackage{latexsym}
\usepackage{graphicx}
\usepackage{pdfpages}
\usepackage {bm}
\usepackage {indentfirst} %indent the first par after section

\usepackage{hyperref,cite}

\advance\hoffset by -1.5cm

\newcommand{\er}{\mathbb{R}}
\newcommand{\B}{\mathcal{B(H)}}
\newcommand{\R}{\mathcal{R}}
\newcommand{\n}{\mathcal{N}}
\newcommand{\h}{\mathcal{H}}

\newcommand{\m}{\mathcal{M}}
\newcommand{\ka}{\mathcal{K}}

\newtheorem{theorem}{Theorem}[section]

\newtheorem{proposition}[theorem]{Proposition}
\newtheorem{corollary}[theorem]{Corollary}

\theoremstyle{definition}

\theoremstyle{definition}
\newtheorem{remark}[theorem]{Remark}

\numberwithin{equation}{section}

\title[Partial isometries and the conjecture of C. K. Fong and S. K. Tsui]{Partial isometries and the conjecture \\ of C. K. Fong and S. K. Tsui}

\author[M. Mbekhta]{Mostafa Mbekhta}
\address{ UFR de Math\'ematiques, Laboratoire Paul Painlev\'e, UMR CNRS 8524,
	Universit\'e Lille 1, 59655 Villeneuve d'Ascq Cedex, France}
\email{mbekhta@math.univ-lille1.fr}

\author[L. Suciu]{Laurian Suciu}
\address{Department of Mathematics, "Lucian Blaga" University
	of Sibiu, Dr. Ion Ra\c tiu 5-7, Sibiu, 550012, Romania}
\email{laurians2002@yahoo.com}

\begin{document}
\begin{abstract}
	We investigate some bounded linear operators $T$ on a Hilbert space which satisfy the condition $|T|\le |{\rm Re} T|$. We describe the maximum invariant subspace for a contraction $T$ on which $T$ is a partial isometry to obtain that, in certain cases, the above condition ensures that $T$ is self-adjoint. In other words we show that the Fong-Tsui conjecture holds for partial isometries, contractive quasi-isometries, or $2$-quasi-isometries, and Brownian isometries of positive covariance, or even for a more general class of operators.	
	\end{abstract}
	\maketitle

\section{Introduction and terminology}
\medskip

For two complex Hilbert spaces $\h$ and $\ka$ we denote by $\mathcal{B}(\h,\ka)$ the Banach space of all bounded linear operators from $\h$ into $\ka$, and $\B=\mathcal{B}(\h,\h)$ considered as a Banach algebra with $I=I_{\h}$ the identity operator on $\h$. For $T\in \mathcal{B}(\h,\ka)$, $\R(T)$ and $\n(T)$ stand for the range and the null-space of $T$, respectively. For a subspace $\mathcal{G}$ of $\h$ its closure is denoted by $\overline{\mathcal{G}}$. As usually, a closed subspace $\mathcal{G}$ of $\h$ is invariant (reducing) for $T$ if $T\mathcal{G} \subset \mathcal{G}$ (and $T^* \mathcal{G} \subset \mathcal{G})$. Also, $T^* \in \mathcal{B}(\ka,\h)$ stands for the adjoint operator of $T$, while the orthogonal projection associated to a closed subspace $\mathcal{G}$ of $\h$ is denoted by $P_{\mathcal{G}}$, that is $P_{\mathcal{G}} \in \mathcal{B}(\h)$ with $P_{\mathcal{G}}^2=P_{\mathcal{G}}=P_{\mathcal{G}}^*$.

An operator $T\in \mathcal{B}(\h,\ka)$ is a {\it contraction} if $\|T\|\le 1$, and $T$ is a {\it partial isometry} when $T^*T$ is an orthogonal projection. In particular $T$ is an {\it isometry} if $T^*T=I_{\h}$, and {\it unitary} if $T$ is a surjective isometry. A unitary operator $U \in \B$ with $U^*=U$ is called a {\it symmetry}. A contraction $T$ is called {\it pure} if $\|Tx\|<\|x\|$ for any $x\in \h$, $x\neq 0$.

Also, we say that an operator $T \in \B$ is a {\it $m$-quasi-isometry} for some integer $m\ge 1$, if $T|_{\R(T^m)}$ is an isometry. The $1$-quasi-isometries are shortly called {\it quasi-isometries}, such operators being firstly studied [18, 19] and latterly in [20, 21], and other articles. The generalization to $m$-quasi-isometries for $m\ge 2$ appear in [8, 15, 16]. It was proved in [16] that a quasi-isometric contraction $T$ is subnormal, that is it has a normal extension, hence $T$ is hyponormal that is $TT^* \le T^*T$.

As usually, for $T\in \B$ we denote the module of $T$ by $|T|=(T^*T)^{1/2}$, and the real part of $T$ by ${\rm Re}T=\frac{1}{2}(T+T^*)$.

An interesting conjecture formulated by C. K. Fong and S. K. Tsui in [11] says that if $T$ satisfies the condition
\begin{align}\label{ec11}
|T|\le |{\rm Re}T|
\end{align}
then $T$ is self-adjoint.

This conjecture was partially proved in [11] in finite dimensional Hilbert spaces, in finite von Neumann algebras, and for compact operators in any Hilbert space.

Originally, C. K. Fong and V. Istr\u{a}\c tescu proved in [10] that $T$ is self-adjoint if and only if $|T|^2\le ({\rm Re}T)^2$. In particular, by Lemma 1.5 [10] it follows that the Fong-Tsui conjecture holds for hyponormal operators.

Recently, M. H. Mortad shows in [17] that this conjecture is also true when $T$ commutes with the partial isometry $U$ which appears in the polar decomposition of ${\rm Re}T$ ($U$ being a symmetry on $\overline{\R({\rm Re}T)}$).

A difficulty for the solution of this conjecture is the fact that $|T|$ and $|{\rm Re}T|$ cannot be easily expressed in terms of $T$ (and $T^*$), by contrast to $|T|^2$ and $({\rm Re}T)^2$. An idea is to transfer the conjecture on some reducing or just invariant parts of $T$ which may be easily expressed in $T, T^*$, and to investigate the condition \eqref{ec11} on each such part.

The purpose of this paper is to show that the Fong-Tsui conjecture is also true for some operators which are related to partial isometries, as well as those before.

In Section 2 we describe the maximum invariant subspace for a contraction $T$ on which $T$ is a partial isometry (Theorem \ref{te21}). We use some block matrix forms for $T$, and we refer to the case when this subspace is just $\n(T^*T-(T^*T)^2)$, and also to a special case when this latter subspace is reducing for $T$.

In Section 3 we give the main result (Theorem \ref{te31}) which asserts that for a contraction $T$ which satisfies the condition \eqref{ec11} the fixed points of $T^*T$, $TT^*$ and $|{\rm Re}T|$ coincide, and their subspace reduces $T$ to a symmetry. So, in this case the invariant partial isometric part of $T$ in $\h$ is $\n(I-T^*T)\oplus \n(T)$, which does not reduce $T$, in general.

As consequences, we derive that the Fong-Tsui conjecture holds for partial isometries, quasi-isometric or $2$-quasi-isometric contractions. In the case of an $m$-quasi-isometric contraction $T$ with $m\ge 3$ we obtain that $T$ is a symmetry on $\overline{\R(T^m)}$, this subspace being even the unitary part of $T$ in $\h$. In this case $T=S\oplus Q$ with $S$ a symmetry on $\overline{\R(T^m)}$ and $Q^m=0$. So $T=T^*$ if and only if $Q=0$.

Another special class of non-contractive operators for which the Fong-Tsui conjecture can be shown true is given by the Brownian isometries. This class containing the Brownian unitaries was extensively studied by J. Agler and M. Stankus in [1-3]. Such operators arise naturally in the context of $2$-isometries, that is of operators $T$ on $\h$ satisfying the identity $T^{*2}T^2-2T^*T+I=0$. According to [2, Proposition 5.37], a $2$-isometry $T$ on $\h$ is a Brownian isometry of covariance $\sigma >0$ if $\sigma ^2 =\|T^*T-I\|$ and, with respect to a decomposition $\h=\h_0\oplus \h_1$, $T$ has a block matrix form
\begin{align}\label{ec12}
T=
\begin{pmatrix}
V & \sigma E\\
0 & U
\end{pmatrix},
\end{align}
where $V$ is an isometry on $\h_0$, $E\in \mathcal{B}(\h_1,\h_0)$ is an injective contraction with $\R(E)\subset \n(V^*)$, while $U$ is unitary on $\h_1$ such that $UE^*E =E^*EU$.

Finally, we remark that, under the condition \eqref{ec11} for a contraction $T$, $T^*$ has a similar block matrix form like that of $T$ (given by Theorem \ref{te31}), without imposing the condition \eqref{ec11} for $T^*$. But these matrix representations of $T$ and $T^*$ cannot lead to a symmetric condition (in $T$ and $T^*$) as $T=T^*$, from a non-symmetric one as \eqref{ec11}. The Fong-Tsui conjecture remains an interesting open problem, in particular for pure contractions.
\medskip

\section{The invariant partial isometric part of a contraction}
\medskip

The main result of this section is the following

\begin{theorem}\label{te21}
For every contraction $T$ on $\h$ there exists the maximum subspace $\m$ which is invariant for $T$, on which $T$ is a partial isometry. More precisely, one has $\n(T) \subset \m \subset \n(T) \oplus \n(I-T^*T)$, and $T$ has the block matrix form
\begin{align}\label{ec21}
T=
\begin{pmatrix}
W & R\\
0 & Q
\end{pmatrix}
\end{align}
on $\h=\m \oplus \m^{\perp} $, where $W$ is a partial isometry on $\m$, $R\in \mathcal{B}(\m^{\perp},\m)$ is a contraction with $W^*R=0$, and $Q$ is a contraction on $\m^{\perp}$ such that $\n(I-Q^*Q)\subset \n(I-T^*T)$.

Moreover, we have
\begin{align}\label{ec22}
\n(T)\oplus \n(I-T^*T)=\m \oplus \n(I-Q^*Q)
\end{align}
if and only if $\n(I-T^*T) \cap \m^{\perp} \subset \n(R)$.
\end{theorem}

\begin{proof}
The required subspace $\m$ need to satisfy the inclusions
$$
\n(T)\subset \m \subset \n(T^*T-(T^*T)^2)=\n(T)\oplus \n(I-T^*T).
$$
But with respect to the decomposition
$$
\h=\n(T)\oplus \n(I-T^*T)\oplus \h'
$$
where $\h'=\overline{\R(T^*T-(T^*T)^2)}$, $T$ has a block matrix form
\begin{align}\label{ec23}
T=
\begin{pmatrix}
0 & T_0 & T_0'\\
0 & T_1 & T_1'\\
0 & T_2 & T_2'
\end{pmatrix},
\end{align}
with some appropriate contractions $T_j$ and $T_j'$. In particular one has $T_2=P_{\h'}T|_{\n(I-T^*T)}$, while the subspace $\h_1:=\n(T_2)$ is invariant for $T_2$. Therefore $\n(T)\oplus \h_1$ is invariant for $T$, and from the above representation of $T$ we can infer another block matrix form for $T$ on the decomposition $\h=\h_0\oplus \h_1 \oplus \h_2$, where $\h_0=\n(T)$ and $\h_2:= [\n(I-T^*T)\ominus \h_1] \oplus \h'$, as follows
\begin{align}\label{ec24}
T=
\begin{pmatrix}
0 & W_0 & R_0\\
0 & W_1 & R_1\\
0 & 0 & Q
\end{pmatrix}.
\end{align}
Here all operators are contractions between the corresponding subspaces of the decomposition of $\h$ and, in particular, for the operator $W\in \mathcal{B}(\h_0,\h_1)$ with the block matrix form
$$
W=
\begin{pmatrix}
0 & W_0\\
0 & W_1
\end{pmatrix},
$$
we have $W_0 =T_0|_{\h_1}=P_{\h_0}T|_{\h_1}$, $W_1=T_1|_{\h_1}=P_{\h_1}T|_{\h_1}$. Also, by considering the operator $R= \begin{pmatrix} R_0\\ R_1 \end{pmatrix} : \h_2 \to \h_0 \oplus \h_1 =: \m$, we can write $T$ in the form
$$
T=
\begin{pmatrix}
W & R\\
0 & Q
\end{pmatrix}
$$
on $\h=\m \oplus \m^{\perp}$. This representation gives
$$
T^*T=
\begin{pmatrix}
W^*W & W^*R\\
R^*W & R^*R+Q^*Q
\end{pmatrix},
$$
and since $\h_1 \subset \n(I-T^*T)$ one has $W^*W|_{\h_1}=I_{\h_1}$. This means
$$
W^*W=
\begin{pmatrix}
0 & 0\\
0 & W_0^*W_0+W_1^*W_1
\end{pmatrix}
=0\oplus I_{\h_1},
$$
that is $W$ is a partial isometry on $\m$. On the other hand, we have
$$
TT^*=
\begin{pmatrix}
WW^*+RR^* & RQ^*\\
QR^* & QQ^*
\end{pmatrix}
\le I,
$$
whence it follows $RR^* \le I_{\m}-P_{\R(W)}=P_{\n(W^*)}$, which means $W^*R=0$.

To show that $\m$ is the maximum invariant subspace for $T$ in $\h$ on which $T$ is a partial isometry, let us consider another such subspace $\m'\subset \h$. So, $T$ has a similar block matrix form on $\h=\m'\oplus (\m')^{\perp}$, namely
$$
T=
\begin{pmatrix}
W' & R'\\
0 & Q'
\end{pmatrix}
$$
with $W'$ a partial isometry on $\m'$ satisfying (as above) $W^{'*}R'=0$. Then
$$
T^*T=
\begin{pmatrix}
W'^*W' & 0\\
0 & R'^*R'+Q'^*Q'
\end{pmatrix}
$$
and this gives
$$
\n(W')\subset \n(T) \subset \m.
$$
As $W'$ is a partial isometry we have the inclusion
$$
\R(W'^*)=\n(I-W'^*W')\subset \n(I-T^*T) \cap \m'
$$
and, in fact, the equality holds here. Indeed, if $x=T^*Tx\in \m'$ and $x\perp \R(W'^*)$ then $0=W'x=TT^*x$ so $Tx=0$ that is $x=0$. Hence we get
$$
\m'=\n(W') \oplus \R(W'^*) \subset \n(T) \oplus \R(W'^*),
$$
which implies
$$
T\R(W'^*) =W'\R(W'^*) \subset \m'\subset \n(T) \oplus \n(I-T^*T)\cap \m'.
$$
We infer (by using the operator $T_2$ in \eqref{ec23}) that
$$
T_2\R(W'^*)=P_{\h'}T|_{\n(I-T^*T)\cap \m'}=0,
$$
which means $\R(W'^*)\subset \n(T_2)=\h_1$. Conclude (by an above inclusion) that
$$
\m'\subset \n(T) \oplus \h_1 =\m,
$$
that is $\m$ has the required maximality property.

Now we prove the other properties concerning the operator $Q$. Firstly, if $x\in \n(I-Q^*Q)$ then
$$
\|x\|=\|Qx\|=\|P_{\m^{\perp}}Tx\| \le \|Tx\|\le \|x\|,
$$
so $x\in \n(I-T^*T)$. This gives $\n(I-Q^*Q) \subset \n(I-T^*T).$

Clearly, the equality \eqref{ec22} is equivalent to the following :
$$
\n(I-T^*T) =\h_1 \oplus \n(I-Q^*Q).
$$
Let $x\in \n(I-T^*T)\cap \m^{\perp}$, so $x=T^*Tx$. Then
$$
Qx=P_{\m^{\perp}}TT^*Tx=P_{\m^{\perp}}Tx
$$
and we have
\begin{eqnarray*}
\|x\|^2 =\|Tx\|^2&=&\|P_{\m}Tx\|^2 + \|P_{\m^{\perp}}Tx\|^2\\
&=& \|Rx\|^2 + \|Qx\|^2.
\end{eqnarray*}
Hence $x\in \n(I-Q^*Q)$ if and only if $Rx=0$. This proves the last assertion of theorem, having in view the inclusion of $\n(I-Q^*Q)$ into $\n(I-T^*T)$ already quoted.
\end{proof}

\begin{corollary}\label{co22}
Let $T$ be a contraction on $\h$. Then the subspace $\n(T) \oplus \n(I-T^*T)$ is invariant for $T$ if and only if $Q$ (in \eqref{ec21}) is a pure contraction on $\m^{\perp}$.
\end{corollary}

\begin{proof}
If $\n(T) \oplus \n(I-T^*T)$ is invariant for $T$ then it is reduced to $\m$, hence $\n(I-Q^*Q)=\{0\}$ by Theorem \ref{te21}. This means that the contraction $Q$ is pure. Conversely, assuming $Q$ pure, we have also $Q^* =T^*|_{\m^{\perp}}$ pure. So, $\m^{\perp}= \overline{\R(I-QQ^*)} \subset \overline{\R(I-TT^*)}$ which implies
$$
T\n(I-T^*T) =\n(I-TT^*) \subset \m \subset \n(T) \oplus \n(I-T^*T).
$$
Hence $\n(T) \oplus \n(I-T^*T)$ is invariant for $T$.
\end{proof}

\begin{remark}\label{re23}
\rm
Recall [14] that the maximum subspace invariant for a contraction $T$ on which $T$ is an isometry is $\n(I-S_T)$, where $S_T$ is the asymptotic limit of $T$, defined as the strong limit of the powers $T^{*n}T^n$, $n\ge 1$. Therefore $\n(I-S_T) \subset \m \cap \n(I-T^*T)$ and, in general, the inclusion is strict, because the powers $T^n$ are not always partial isometries on $\m$.

In the case that $\m=\n(T) \oplus \n(I-T^*T)$ and $\n(T^*)$ is invariant for $T$ (that is $\n(T^*) \subset \n(T)$), and if $\m^* \subset \h$ is the corresponding subspace for $T^*$ given by Theorem \ref{te21}, then $\m^* \subset \n(T^*)\oplus \n(I-TT^*)=\n(T^*) \oplus T\n(I-T^*T) \subset \m$. In addition, if $\n(I-TT^*)$ is invariant for $T$ then
$$
T\m^* =\n(T^*) \oplus T\n(I-TT^*) \subset \n(T^*) \oplus \n(I-TT^*)=\m^*,
$$
hence $\m^* $ reduces $T$ to a partial isometry.

Clearly, the maximum subspace which reduces $T$ to a partial isometry exists always, but it is different of $\m\cap \m^*$, in general. Its structure is more complicated, and will not be given here.

\end{remark}

In the following section we see that the subspace $\m$ has the form from Corollary \ref{co22} with $\n(I-T^*T)$ reducing for $T$, under the condition \eqref{ec11}, but $\m$ does not reduce $T$, in general.

\medskip

\section{On the Fong-Tsui conjecture}
\medskip

Remark firstly that in the Fong-Tsui conjecture one can suppose that $T$ is a contraction, because the condition \eqref{ec11} works simultaneously for $T$ and $T|_{\|T\|}$.

Concerning the structure of such a contraction we have the following main result.

\begin{theorem}\label{te31}
Let $T$ be a contraction on $\h$ satisfying the condition \eqref{ec11}. Then
\begin{align}\label{ec31}
\n(I-T^*T)=\n(I-TT^*)=\n(I-|{\rm Re}T|)
\end{align}
and this subspace reduces $T$ to a symmetry. Also, we have
\begin{align}\label{ec32}
\n({\rm Re}T)=\n(T)\cap \n(T^*), \quad \n(T^*)=\n({\rm Re}T)\oplus \n(T^*|_{\overline{\R(T^*)}}).
\end{align}

Moreover, one has $\n(T)=\n({\rm Re}T)$ if and only if $T=U\oplus Z$ with respect to a decomposition $\h=\mathcal{G} \oplus \mathcal{G}^{\perp}$, where $U$ is a symmetry, and $Z$ is a pure contraction satisfying the condition \eqref{ec11}.
\end{theorem}

\begin{proof}
Consider the block matrix \eqref{ec24} of $T$ that is
$$
T=
\begin{pmatrix}
0 & W_0 & R_0\\
0 & W_1 & R_1\\
0 & 0 & Q
\end{pmatrix}
$$
on $\h=\h_0 \oplus \h_1 \oplus \h_2$, where $\h_0:= \n(T)$, $\h_1:= \n(P'T|_{\n(I-T^*T)})$, while $P'$ is the orthogonal projection onto $\h':= \overline{\R(T^*T-(T^*T)^2)}$, and $\h_2= [\n(I-T^*T) \ominus \h_1] \oplus \h'$. In addition, from the proof of Theorem \ref{te21} we have the relations
\begin{align}\label{ec33}
W_0^*W_0 + W_1^*W_1=I, \quad W_0^*R_0 + W_1^*R_1=0,
\end{align}
and
\begin{align}\label{ec34}
W_0W_0^* +R_0R_0^*\le I, \quad W_1W_1^*+R_1R_1^*\le I.
\end{align}

By a simple computation we get the representations :
$$
T^*T=
\begin{pmatrix}
0 & 0 & 0\\
0 & I & 0\\
0 & 0 & R_0^*R_0+R_1^*R_1+ Q^*Q
\end{pmatrix},
\quad {\rm Re}T=\frac{1}{2}
\begin{pmatrix}
0 & W_0 & R_0\\
W_0^* & 2 {\rm Re}W_1 & R_1\\
R_0^* & R_1^* & 2 {\rm Re}Q
\end{pmatrix}
$$
and respectively (by using the second relation of \eqref{ec33})
$$
({\rm Re}T)^2=\frac{1}{4}
\begin{pmatrix}
W_0W_0^* +R_0R_0^* & 2W_0 {\rm Re}W_1 +R_0R_1^* & W_0R_1 +2 R_0 {\rm Re}Q\\
2 ({\rm Re}W_1)W_0^*+R_1R_0^* & W_0^*W_0 + 4 ({\rm Re}W_1)^2 + R_1R_1^* & W_1R_1 + 2 R_1 {\rm Re}Q\\
R_1^* W_0^*+ 2 ({\rm Re}Q)R_0^* & R_1^* W_1^* + 2 ({\rm Re}Q)R_1^* & R_0^*R_0 +R_1^*R_1 + 4 ({\rm Re}Q)^2
\end{pmatrix}.
$$
Since the subspace $\h_1$ reduces $|T|$ to $I_{\h_1}$ and as ${\rm Re}T$ is a contraction, the condition \eqref{ec11} implies that $\h_1$ also reduces $|{\rm Re}T|$ and $|{\rm Re}T||_{\h_1}=I_{\h_1}$. As $|{\rm Re}T|^2=({\rm Re}T)^2$ we get from the above matrix representation that
\begin{align}\label{ec35}
\frac{1}{4} (W_0^*W_0+4({\rm Re}W_1)^2 + R_1R_1^*)=I_{\h_1}.
\end{align}
By using the former relation in \eqref{ec33} and the second relation in \eqref{ec34} we obtain
$$
I_{\h_1}\le {\rm Re} W_1^2,
$$
which means by [11, Corollary 3] that $W_1^2=I$. So, the relation \eqref{ec35} becomes $W_1W_1^*+R_1R_1^*=I_{\h_1}$ which leads to $I_{\h_1} +W_1R_1R_1^*W_1^*=W_1W_1^*$. But this gives $R_1=0$ and $W_1W_1^*=I_{\h_1}$, hence $W_1=W_1^*$, while by \eqref{ec33} this yields $W_0=0$. So, the block matrices of $T$, $T^*T$, ${\rm Re}T$ and $({\rm Re}T)^2$ have simpler forms.

Now, we have by Theorem \ref{te21} that $\n(I-Q^*Q)\subset \n(I-T^*T)\cap \h_2$, and we show next that $\n(I-T^*T) \cap \h_2 \subset \n(I-QQ^*)$. Indeed, let $x=T^*Tx \in \h_2$. To use the condition \eqref{ec11} we consider $|{\rm Re}T|$ to have the following block matrix form on $\h=\h_0\oplus \h_1\oplus \h_2$ (having in view that $|{\rm Re}T||_{\h_1}=I_{\h_1}$) :
$$
|{\rm Re}T|=
\begin{pmatrix}
A & 0 & B\\
0 & I & 0\\
B^* & 0 & C
\end{pmatrix},
$$
with some appropriate contractions $A, B$ and $C$ with $A,C\ge 0$. As $|{\rm Re}T|^2=({\rm Re}T)^2$ we obtain that $B^*B+C^2=\frac{1}{4}[R_0^*R_0 + (Q + Q^*)^2]$, so for $x$ as above we get (by \eqref{ec11})
$$
\|x\|^2 =\langle |T|x,x \rangle \le \langle |{\rm Re}T|x,x\rangle = \langle Cx,x \rangle \le \|x\|^2
$$
which means $Cx=x$ (because $C\ge 0$). Then the above equality together with the fact that $(R_0^*R_0+Q^*Q)x=T^*Tx=x$, lead to the relation
$$
\frac{3}{4} \|x\|^2 + \|Bx\|^2 =\frac{1}{4} \|Q^*x\|^2 + \frac{1}{2} \langle {\rm Re}Q^2x,x \rangle.
$$
Since $Q^2$ is a contraction it follows that
$$
\frac{1}{4} (\|x\|^2 - \|Q^*x\|^2) + \|Bx\|^2 \le 0,
$$
hence $Bx=0$ and $\|Q^*x\|=\|x\|$. So $x\in \n(I-QQ^*)$ and the inclusion $\n(I-T^*T) \cap \h_2 \subset \n(I-QQ^*)$ is proved.

Next, we consider the block matrix form of $TT^*$ on $\h=\h_0 \oplus \h_1 \oplus \h_2$, namely
$$
TT^*=
\begin{pmatrix}
R_0R_0^* & 0 & R_0Q^*\\
0 & I & 0 \\
QR_0^* & 0 & QQ^*
\end{pmatrix}.
$$
Since $T^*|_{\h_2}=Q^*$ and $T$ is a contraction we have $\n(I-QQ^*) \subset \n(I-TT^*)$, and from this representation of $TT^*$ one has also $\h_1 \subset \n(I-TT^*)$. So we infer that
\begin{align}\label{ec36}
\n(I-T^*T) =\h_1 \oplus \n(I-T^*T) \cap \h_2 \subset \h_1 \oplus \n(I-QQ^*) \subset \n(I-TT^*).
\end{align}
This means that $\n(I-TT^*)$ is invariant for $T^*$, hence $\n(I-T^*T)$ is also invariant for $T^*$.

To see that $\n(I-T^*T)$ just reduces $T$ we firstly remark from \eqref{ec11} that
$$
\n(I-T^*T) =\n(I- |T|) \subset \n(I- |{\rm Re}T|)=\n(I- ({\rm Re}T)^2).
$$
In fact, these subspaces coincide, they containing the subspace $\h_1$. To see this equality, let us consider the polar decomposition
$$
{\rm Re}T=\widetilde{U}|{\rm Re}T|,
$$
where $\widetilde{U}$ is a symmetry on $\overline{\R({\rm Re}T)}=\overline{\R(|{\rm Re}T|)}$. Clearly, one has
$$
\n(I-|{\rm Re}T|)=\n(\widetilde{U}-{\rm Re}T)
$$
and this subspace reduces the operators ${\rm Re}T$ and $\widetilde{U}$.

Let $x\in \n(\widetilde{U}-{\rm Re}T) \cap \h_2$ such that $x$ is orthogonal on $\n(I-T^*T)$, hence $\|Tx\|<\|x\|$. As $\h_2 \subset \overline{\R({\rm Re}T)}$ and $\widetilde{U}$ is unitary on this range, we get
$$
\|x\|=\|Ux\|=\|({\rm Re}T)x\|\le \frac{1}{2} (\|Tx\|+\|T^*x\|)< \|x\|
$$
which forces to have $x=0$. Therefore $\n(I-|{\rm Re}T|) \cap \h_2 \subset \n(I-T^*T)$, and since $\h_1 \subset \n(I-|{\rm Re}T|)$ we conclude that $\n(I-T^*T) =\n(I-|{\rm Re}T|)$. But this subspace reduces ${\rm Re}T$ and it is invariant for $T^*$ (as we have seen before). Hence $\n(I-T^*T)$ reduces $T$, which also gives the inclusion $\n(I-TT^*) \subset \n(I-T^*T)$. Finally, we conclude that
$$
\n(I-T^*T) =\n(I-TT^*)=\n(I- |{\rm Re}T|)=\h_1,
$$
where for the last equality we have in view the maximality of the subspace $\n(T) \oplus \h_1$ relative to $T$. The identities \eqref{ec31} are proved.

In addition, from the inclusions \eqref{ec36} we infer that $\n(I-QQ^*)=\{0\}$, so $Q^*$ like $Q$ are pure contractions on $\h_2$.

Now, the condition \eqref{ec11} yields $\n({\rm Re}T)=\n(T)\cap \n(T^*)$, and this subspace reduces $T$. Therefore we have $\n({\rm Re}T)=\n(T)$ if and only if $\n(T)$ reduces $T$, that is $R_0=0$. Equivalently, this means that $T$ has the diagonal representation $T=W_1\oplus 0\oplus Q$ on $\h=\h_1\oplus \h_0 \oplus \h_2$, with $W_1, Q$ as above, that is $T=W_1\oplus Z$ where $Z=0\oplus Q$ is a pure contraction which satisfies the condition \eqref{ec11}.

Finally, for the second relation in \eqref{ec32} we have from the above block matrix form of $T$ (with $W_0=0$, $R_1=0$)
$$
\n(T^*)=\{y\oplus z \in \n(T) \oplus \h_2 : R_0^*y+Q^* z=0\}.
$$
But $\n(T^*) \cap \n(T)=\n(R_0^*)$ and $\n(T^*|_{\h_2})=\n(Q^*)=\n(T^*)\cap \h_2= \n(T^*|_{\overline{\R(T^*)}})$. So we obtain
$$
\n(T^*)=\n(R_0^*)\oplus \n(Q^*) =\n({\rm Re}T)\oplus \n(T^*|_{\overline{\R(T^*)}}),
$$
and this finishes the proof.
\end{proof}

\begin{corollary}\label{co32}
A real scalar multiple of a partial isometry which satisfies the condition \eqref{ec11} is self-adjoint.
\end{corollary}

\begin{proof}
If $T$ is a partial isometry satisfying \eqref{ec11} then $\h=\n(I-T^*T) \oplus \n(T)$, so $T=U\oplus 0=T^*$ by Theorem \ref{te31}. More general, if $T=\alpha T_0$ with $\alpha \in \er$ and $T_0$ a partial isometry, then $\frac{1}{\alpha}T=\frac{1}{\alpha}T^*$ by the previous remark, so $T=T^*$.
\end{proof}

A more general result than the previous corollary is the following

\begin{proposition}\label{pr33}
Let $0 \neq T \in \B$ having with respect to a decomposition $\h=\mathcal{G} \oplus \mathcal{G}^{\perp}$ the block matrix form
\begin{align}\label{ec37}
T=
\begin{pmatrix}
S & R\\
0 & Q
\end{pmatrix},
\end{align}
where $\frac{1}{\|T\|}S$ is an isometry on $\mathcal{G}$ and, in addition, either $R\in \mathcal{B}(\mathcal{G}^{\perp}, \mathcal{G})$ is injective and $Q\in \mathcal{B}(\mathcal{G}^{\perp})$ is arbitrary, or $Q^2=0$. If $T$ satisfies the condition \eqref{ec11} then $T$ is self-adjoint.
\end{proposition}

\begin{proof}
Let $T \neq 0$ as in \eqref{ec37}, and let $T_0 =\frac{1}{\alpha}T$, $S_0=\frac{1}{\alpha}S$, $R_0=\frac{1}{\alpha}R$, $Q_0=\frac{1}{\alpha}Q$, where $\alpha =\|T\|$. Then the contraction $T_0$ satisfies \eqref{ec11}, and by Theorem \ref{te31} we infer $\n(I-T_0^*T_0)=\n(I-T_0T_0^*)$. To use this fact we have in view the representations
$$
I-T_0^*T_0=
\begin{pmatrix}
0 & 0\\
0 & I-R_0^*R_0-Q_0^*Q_0
\end{pmatrix},
\quad I-T_0T_0^*=
\begin{pmatrix}
I- S_0S_0^*-R_0R_0^* & -R_0Q_0^*\\
-Q_0R_0^* & I-Q_0Q_0^*
\end{pmatrix},
$$
where for $I-T_0^*T_0$ we used that $S_0$ is an isometry on $\mathcal{G}$. These representations together with the above kernels give that $\mathcal{G} \subset \n(I-T_0T_0^*)$, that is the relations
\begin{align}\label{ec38}
I-S_0S_0^*-R_0R_0^*=0, \quad Q_0R_0^*=0.
\end{align}

From the first relation we have $R_0R_0^*=P_{\n(S_0^*)}$, so $R_0$ is a partial isometry and $\R(R_0)=\n(S_0^*)$.

Now, if $R$ is injective, then from the second relation in \eqref{ec38} one obtains $Q_0=0$. In this case $R_0$ is an isometry on $\mathcal{G}^{\perp}$, and since by the above matrix representations we get
$$
\mathcal{G}=\n(I-T_0T_0^*)=\n(I-T_0^*T_0)=\mathcal{G}\oplus \n(I-R_0^*R_0)=\mathcal{G}\oplus \mathcal{G}^{\perp},
$$
it follows $\mathcal{G}^{\perp}=\{0\}$. Hence $T_0=S_0$ is a symmetry on $\h$, and consequently $T=T^*$.

Assume next the other condition from hypothesis, namely for $Q^2=0$. Since $\n(I-T_0^*T_0)=\mathcal{G}\oplus \n(I-Q_0Q_0^*)$ is invariant for $T_0^*$, $\n(I-Q_0Q_0^*)$ will be invariant for $Q_0^*=T_0^*|_{\mathcal{G}^{\perp}}$. So, if $b\in \n(I-Q_0Q_0^*)$ we have $Q_0^*b=Q_0Q_0^{*2}b=0$, which means $b=Q_0Q_0^*b\in \n(Q_0^*)$ that is $b=0$. Hence $\n(I-Q_0Q_0^*)=\{0\}$ which gives $\n(I-T_0T_0^*)=\mathcal{G}$ and this subspace reduces $T_0$. Thus $R_0=0$, while $S_0$ and $Q_0$ satisfy the condition \eqref{ec11}. Therefore $S_0$ will be a symmetry on $\mathcal{G}$, and as $Q_0^2=0$ it is easy to see that $Q=0$. Indeed, since $Q^2=0$, $Q$ will have on $\mathcal{G}^{\perp}=\overline{\R(Q)}\oplus \n(Q^*)$ the block matrix form
$$
Q=
\begin{pmatrix}
0 & Q_1\\
0 & 0
\end{pmatrix}.
$$

By a simple computation we get
$$
|Q|=0\oplus |Q_1|, \quad |{\rm Re}Q|=\frac{1}{2} (|Q_1^*|\oplus |Q_1|)
$$
on the above decomposition of $\mathcal{G}^{\perp}$. So, the condition \eqref{ec11} for $Q$ implies $Q_1=0$ that is $Q=0$.

We conclude that $T=S \oplus 0=\alpha S_0\oplus 0=T^*$, $S_0$ being a symmetry. This ends the proof.
\end{proof}

To apply this proposition for $2$-quasi-isometries, we recall (see [8, Remark 3.10]; or [15, Remark 2.7]) that such an operator has a block matrix form as in \eqref{ec37} on $\h=\overline{\R(T^2)}\oplus \n(T^{*2})$, with $S$ an isometry and $Q^2=0$. So the previous proposition gives the following

\begin{corollary}\label{co34}
A contractive $2$-quasi-isometry which satisfies the condition \eqref{ec11} is self-adjoint.
\end{corollary}

It is still unknown if the second assumption in the hypothesis of Proposition \ref{pr33}, namely $Q^2=0$, can be replaced by the weaker condition $Q^m=0$ for some $m\ge 3$, in order to preserve the conclusion; that is to prove that $Q=0$ under the condition \eqref{ec11}.

However, for $m$-quasi-isometries we have the following

\begin{corollary}\label{co35}
If $T\in \B$ is a contractive $m$-quasi-isometry for an integer $m\ge 3$ which satisfies the condition \eqref{ec11} then $\n(I-T^*T)=\n(I-TT^*)=\overline{\R(T^m)}$ reduces $T$ to a symmetry. In addition, $T$ is self-adjoint if and only if $T=S\oplus 0$.
\end{corollary}

\begin{proof}
If $T$ is a $m$-quasi-isometry then $T|_{\R(T^m)}$ is an isometry. Such an operator has the form \eqref{ec37} on $\h= \overline{\R(T^m)} \oplus \n(T^{*m})$ with $S$ an isometry and $Q^m=0$. In the case when $\|T\|=1$ we also have (as in the proof of Proposition \ref{pr33}) that $S^*R=0$ and
  $$
  \n(I-T^*T)=\n(I-TT^*)= \overline{\R(T^m)} \oplus \n(I-QQ^*),
  $$
this subspace reducing $T$ to a symmetry. So, if $b\in \n(I-QQ^*)$ and $m\ge 3$ then $Q^*b \in \n(I-QQ^*)$ which leads (by recurrence) to $b=QQ^*b=Q^mQ^{*m}b=0$. Hence $\n(I-QQ^*)=\{0\}$, $\n(I-T^*T)=\overline{\R(T^m)}$, and $S$ is a symmetry on this subspace. Then the condition $S^*R=0$ yields $R=0$, consequently $T=S\oplus Q$.

Assume now $T=T^*$ that is $Q=Q^*$. Since $Q^m=0$ one has $\R(Q^{m-1})\subset \n(Q)=\n(Q^*) \subset \n(Q^{*(m-1)})$, hence $Q^{m-1}=0$. By recurrence one infers $Q=0$, so $T=S \oplus 0$. The converse implication for the second assertion of corollary being trivial, the proof is finished.
\end{proof}

We remarked before that a non-null nilpotent operator of order $2$ cannot satisfy the condition \eqref{ec11}. We can also use this fact to obtain the following result

\begin{proposition}\label{pr36}
Let $T$ be a contraction on $\h$ having with respect to the decomposition $\h=\mathcal{G} \oplus \mathcal{G}^{\perp}$ the block matrix form
\begin{align}\label{ec39}
T=
\begin{pmatrix}
W & R\\
0 & W'
\end{pmatrix}
\end{align}
where $W$ and $W'$ are partial isometries on $\mathcal{G}$ and $\mathcal{G}^{\perp}$, respectively, while $R\in \mathcal{B}(\mathcal{G}^{\perp}, \mathcal{G})$. If $T$ satisfies the condition \eqref{ec11} then $T$ is self-adjoint.
\end{proposition}

\begin{proof}
 Consider firstly that $W'=0$ in \eqref{ec39}. Then $\overline{\R(T)} \subset \mathcal{G}$ and, assuming \eqref{ec11}, by Theorem \ref{te21} and Theorem \ref{te31} one has $\mathcal{G} \subset \n(T) \oplus \n(I-T^*T)=:\m$. Therefore $\m^{\perp}\subset \mathcal{G}^{\perp}\subset \n(T^*)$ that is $T^*|_{\m^{\perp}}=0$, and using the $3 \times 3$ block matrix of $T$ on $\h=\n(T) \oplus \n(I-T^*T) \oplus \m^{\perp}$ given by the proof of Theorem \ref{te31} (with $W_0=0$, $R_1=0$, $Q=0$ by the previous remark, $W_1$ a symmetry on $\n(I-T^*T)$, and $R_0 \in \mathcal{B}(\m^{\perp}, \n(T))$) we get the representation $T=W_1 \oplus \widetilde{R}$ on $\h=\n(I-T^*T) \oplus [\n(T)\oplus \m^{\perp}]$, where $\widetilde{R} \in \mathcal{B}(\n(T) \oplus \m^{\perp})$ has the block matrix form
$$
\widetilde{R}=
\begin{pmatrix}
0 & R_0\\
0 & 0
\end{pmatrix}.
$$
Since $\widetilde{R}^2=0$ and $\widetilde{R}$ satisfies \eqref{ec11} one has $\widetilde{R}=0$, which ensures $T=T^*$.

In the general case, as $W'$ is a partial isometry on $\mathcal{G}^{\perp}$ and $\n(I-T^*T)=\n(I-TT^*)$ by \eqref{ec31}, we have $\mathcal{G}^{\perp}\subset \n(I-T^*T)\oplus \n(T^*)$, and as above $\mathcal{G} \subset \m$. So, $\m^{\perp}\subset \mathcal{G}^{\perp}$ which by the previous inclusion of $\mathcal{G}^{\perp}$ gives $\m^{\perp} \subset \n(T^*)$. This shows that $T$ has on $\h=\m \oplus \m^{\perp}$ a block matrix of the form
$$
T=
\begin{pmatrix}
\widetilde{W} & \widetilde{R}\\
0 & 0
\end{pmatrix},
$$
with $\widetilde{W}=0\oplus W_1$ a partial isometry and then it follows that $T=T^*$ by the previous conclusion.
\end{proof}

The contractions mentioned in this proposition are not hyponormal, in general. But among these one gets the quasi-isometries (the case when $W$ is an isometry), which are subnormal (as we already quoted in the introduction).

\begin{remark}\label{re37}
\rm
An operator $T\in \B$ having the block matrix form \eqref{ec37} as in Proposition \ref{pr33} is not a contraction, in general. But in the case when $T$ is a contraction with $S$ an isometry and $R$, $Q$ arbitrary contractions (in \eqref{ec37}), then the condition $S^*R=0$ is also true and $R$ will be a partial isometry (as in the proof of Proposition \ref{pr33}), if $T$ satisfies \eqref{ec11}. In general $R\neq 0$, but if $Q$ is pure then $R=0$ because $\n(I-T^*T)$ reduces $T$ to a symmetry and we have
$$
\n(I-T^*T)=\n(I-TT^*)=\mathcal{G} \oplus \n(I-R^*R-Q^*Q)=\mathcal{G} \oplus \n(I-QQ^*).
$$
This latter case occurs, for instance, when $\mathcal{G}=\n(I-S_T)$ in the representation \eqref{ec37} of $T$, and then the Fong-Tsui conjecture for $T$ one reduces to its pure part $Q$.

In fact, every operator $T\neq 0$ on $\h$ has the form \eqref{ec37} with $R,Q$ arbitrary operators, on the decomposition $\h=\n(I-S_{T_0})\oplus \overline{\R(I-S_{T_0})}$, where $T_0=\frac{1}{\|T\|}T$, $S_{T_0}$ is the asymptotic limit of $T_0$, and so $\frac{1}{\|T\|}S$ is an isometry (in \eqref{ec37}). Thus, the assumptions on the operators $R$ or $Q$ in Proposition \ref{pr33} and the condition \eqref{ec11} force such operator $T$ to be self-adjoint.

In turn to Proposition \ref{pr36}, it is clear that in \eqref{ec39} one can consider multiples of partial isometries with the scalar $\alpha = \|T\|$ instead of $W$ and $W'$ respectively, in order to preserve the conclusion.

Another special class of non-contractive operators which contains some $2$-isometries as well as the Brownian isometries, and for which the Fong-Tsui conjecture holds, is mentioned by the following
\end{remark}

\begin{proposition}\label{pr38}
Let $T\in \B$ such that $T^*T\ge I$ and having the block matrix form \eqref{ec37} on $\h=\mathcal{G} \oplus \mathcal{G}^{\perp}$, where $S$ is an isometry on $\mathcal{G}$, $R\in \mathcal{B}(\mathcal{G}^{\perp}, \mathcal{G})$ with $S^*R=0$, while $Q$ is a contraction on $\mathcal{G}^{\perp}$. If $T$ satisfies the condition \eqref{ec11} then $T$ is self-adjoint, in fact a symmetry.
\end{proposition}

\begin{proof}
By using \eqref{ec37} with $S,R,Q$ as above one obtains
$$
T^*T=
\begin{pmatrix}
I & 0\\
0 & R^*R+Q^*Q
\end{pmatrix},
\quad ({\rm Re}T)^2= \frac{1}{4}
\begin{pmatrix}
(S+S^*)^2+RR^* & SR+R(Q+Q^*)\\
R^*S^* + (Q+Q^*)R^* & R^*R+ (Q+Q^*)^2
\end{pmatrix}.
$$

Since $T^*T \ge I$ the condition $|T|\le |{\rm Re}T|$ implies $|{\rm Re}T|\ge I$. Therefore we have $|T|\le |{\rm Re}T|\le ({\rm Re}T)^2$ which gives for any $x\in \mathcal{G}$ (as $R^*S=0$),
\begin{eqnarray*}
\|x\|^2&=& \|Sx\|^2= \langle |T|Sx,Sx \rangle \le \langle ({\rm Re}T)^2 Sx,Sx\rangle \\
&=& \frac{1}{4} \langle ((S+S^*)^2 +RR^*)Sx,Sx \rangle =\frac{1}{4} \langle (S^3+S^* +2S)x, Sx\rangle \\
& =& \frac{1}{2} (\|x\|^2 + \langle ({\rm Re}S^2)x,x \rangle ).
\end{eqnarray*}
This means that $I_{\mathcal{G}} \le {\rm Re}S^2$, which by [10, Corollary 3] yields $S^2=I_{\mathcal{G}}$. As $S$ is an isometry it will be just a symmetry, and as $S^*R=0$ it follows $R=0$.

On the other hand, since $I\le ({\rm Re}T)^2$ we have for each $y \in \mathcal{G}^{\perp}$ ($Q$ being a contraction)
$$
\|y\|^2 \le \frac{1}{4}\langle (Q+Q^*)^2y,y \rangle \le \frac{1}{4} \langle (Q^2+Q^{*2} +2I)y,y \rangle,
$$
whence $I_{\mathcal{G}^{\perp}}\le {\rm Re} Q^2$. Then as above $Q$ is a symmetry, hence $T=S \oplus Q=T^*$. This ends the proof.
\end{proof}

In concordance with the representations \eqref{ec12} and \eqref{ec37} for Brownian isometries, from the previous proposition we derive the following

\begin{corollary}\label{co39}
A Brownian isometry of positive covariance which satisfies the condition \eqref{ec11} is a symmetry.
\end{corollary}

\begin{remark}\label{re310}
\rm
Proposition \ref{pr38} one refers to a larger class of operators than that of Brownian isometries, but not to all $2$-isometries, because we need to impose that $Q$ is a contraction.

Recall (see [1, Theorem 1.26]) that the block matrix of a $2$-isometry on $\h=\n(T^*T-I) \oplus \overline{\R(T^*T-I)}$ has the form \eqref{ec37}, with $S$ an isometry, $S^*R=0$, $R^*R+Q^*Q-I$ injective and
$$
Q^*(R^*R+Q^*Q-I)Q=R^*R+Q^*Q-I.
$$
But if such $T$ satisfies the condition \eqref{ec11}, then as in the previous proof $S$ will be a symmetry, hence $R=0$. In this case, the last relation before means $Q^*(Q^*Q-I)Q=Q^*Q-I$, that is $Q$ is a $2$-isometry and $Q^*Q\ge I$. So, the Fong-Tsui conjecture for $T$ one reduces to its $2$-isometric pure part $Q$.
\end{remark}

\medskip

\section{Final Remarks}

\medskip

In turn to Theorem \ref{te31} which plays an essential role for our considerations concerning the Fong-Tsui conjecture, we make some comments.

\begin{remark}\label{re41}
\rm
The operator $T^*$ has a similar form like $T$, under the condition \eqref{ec11}, and in the corresponding decomposition of $\h$. Indeed, we have from the proof of Theorem \ref{te31} the representation
\begin{align}\label{ec41}
T=
\begin{pmatrix}
W_1 & 0 & 0\\
0 & 0 & R_0\\
0 & 0 & Q
\end{pmatrix}
\end{align}
on the decomposition $\h=\h_1 \oplus \h_0 \oplus \h_2$ where $\h_1=\n(I-T^*T)$ and $\h_0=\n(T)$. To see here $\n(T^*)$ given by \eqref{ec32}, we refine \eqref{ec41} by considering $\n(T)=\n({\rm Re}T) \oplus \h_0'$ and $\h_2=\n(Q^*)\oplus \overline{\R(Q)}$. So, on the decomposition $\h=\h_1 \oplus \n({\rm Re}T)\oplus \n(Q^*)\oplus \h_0'\oplus \overline{\R(Q)}$ we have the representations
$$
T=
\begin{pmatrix}
W_1 & 0 & 0 & 0 & 0\\
0 & 0 & 0 & 0 & 0\\
0 & 0 & 0 & 0 & 0\\
0 & 0 & R_{00} & 0 & R_{01}\\
0 & 0 & Q_0 & 0 & Q_1
\end{pmatrix},
\quad T^*=
\begin{pmatrix}
W_1 & 0 & 0 & 0 & 0\\
0 & 0 & 0 & 0 & 0 \\
0 & 0 & 0 & R_{00}^* & Q_0^*\\
0 & 0 & 0 & 0 & 0\\
0 & 0 & 0 & R_{01}^* & Q_1 ^*
\end{pmatrix},
$$
where $R_{00}=P_{\h_0'}R_0|_{\n(Q^*)}$, $R_{01}=P_{\h_0'}R_0|_{\overline{\R(Q)}}$, and $Q_0=Q|_{\n(Q^*)}$, $Q_1=Q|_{\overline{\R(Q)}}$. Here we used the fact that $P_{\n({\rm Re}T)}R_0=0$ and $P_{\n(Q^*)}Q=0$. We get that $T^*$ has on $\h=\h_1 \oplus \n(T^*) \oplus \h_3$ with $\h_3=\h_0'\oplus \h_3$ the block matrix form
$$
T^*=
\begin{pmatrix}
W_1 & 0 & 0\\
0 & 0 & R_*\\
0 & 0 & Q_*
\end{pmatrix},
$$
where $R_*:= \begin{pmatrix} 0 & 0 \\ R_{00}^* & Q_0^* \end{pmatrix}$ and $Q_*:= \begin{pmatrix} 0 & 0 \\ R_{01}^* & Q_1^* \end{pmatrix}.$

Hence $T^*$ has the same form as $T$ in \eqref{ec41} on the corresponding decomposition $\h=\h_1 \oplus \n(T^*) \oplus \h_3$, which was obtained under the condition \eqref{ec11}, even if $T^*$ does not satisfy this condition.

These representations of $T$ and $T^*$ also provide that the invariant partial isometric parts for $T$ and $T^*$ can be different under the condition \eqref{ec11}, that is $\n(T) \neq \n(T^*)$, in general. But the condition $\n(T)=\n(T^*)$ is necessary for $T$ to be self-adjoint. In this latter case we have $R_0=0$ and $Q_0=0$, while $Q=Q_1$ like $Q^*$ are injective and pure contractions, hence
$$
\overline{\R(Q)}=\overline{\R(Q^*)}=\overline{\R(I-Q^*Q)}=\overline{\R(I-QQ^*)}.
$$
But only these information on $Q$ are not sufficient to obtain that $Q=Q^*$ that is $T=T^*$.

The major difficulty to use the condition \eqref{ec11} in order to obtain the self-adjointness of $T$ consists in the fact that $|{\rm Re}T|=[({\rm Re}T)^2]^{1/2}$ cannot be easily expressed in terms of $T$ (that is using the block matrix form \eqref{ec41}).

Clearly, on can have in view the polar decomposition ${\rm Re}T=U|{\rm Re}T|$, where $U$ is a symmetry on $\overline{\R({\rm Re}T)}$. Since $\n({\rm Re}T)=\n(U)$ reduces $U$ and $T$, one can just only consider the condition \eqref{ec11} on $\overline{\R({\rm Re}T)}$ which reduces $T$ and $U$.

But even in the case when $\n(T)$ reduces $T$ ($R_0=0$ in \eqref{ec39}), hence when $U$ is a symmetry on the subspace $\h_2$, we cannot use this to obtain $Q=Q^*$.

But, as we have seen before, this is possible for some classes of operators related to partial isometries, or for the operators $T$ which commute with $U$, as it was recently shown in [17]. Let us remark that this last class of operators and that of partial isometries are not contained one into the other. For example, one can consider the partial isometry $T\in \mathcal{B}(\h\oplus \h)$ and the corresponding partial isometry for ${\rm Re}T$ with the block matrices
$$
T=
\begin{pmatrix}
0 & I\\
0 & 0
\end{pmatrix},
\quad U=
\begin{pmatrix}
0 & I\\
I & 0
\end{pmatrix},
$$
and it is clear that $TU=I \oplus 0$ and $UT=0 \oplus I$.
\end{remark}

\begin{remark}\label{re42}
\rm
The condition \eqref{ec11} means that there exists a contraction $A$ on $\h$ satisfying the relation
\begin{align}\label{ec42}
A|{\rm Re}T|^{1/2}=|T|^{1/2},
\end{align}
or equivalently
\begin{align}\label{ec43}
|{\rm Re}T|^{1/2}A^*=|T|^{1/2}.
\end{align}
From these relations we have $\n(A^*) \subset \n(T)$ and also
$$
A\overline{\R({\rm Re}T)}\subset \overline{\R(T^*)} \subset \overline{\R({\rm Re}T)}
$$
that is $\overline{\R({\rm Re}T)}$ is invariant for $A$. Denoting
$$
A_0=A|_{\overline{\R({\rm Re}T)}},
$$
we infer that
$$
\overline{\R(T^*)} =\overline{A|{\rm Re}T|^{1/2}\h} =\overline{A_0\R({\rm Re}T)}= \overline{\R(A_0)}\subset \overline{\R(A)}.
$$

In addition, if $\n(T)$ reduces $T$ then
$$
\overline{\R(A)}=\overline{\R(A_0)}=\overline{\R(T^*)}=\overline{\R({\rm Re}T)},
$$
and some relationship between the closeness of the ranges of $T$, ${\rm Re}T$ and $A$ can be obtained, but unessential for the Fong-Tsui conjecture. However, from \eqref{ec42} we infer $|T|^2 \le A |{\rm Re}T|^2A^*$. So, if $A^*$ is a $|{\rm Re} T|^2$-contraction that is $A|{\rm Re}T|^2A^*\le |{\rm Re}T|^2$, then $|T|^2 \le |{\rm Re}T|^2$ and by [10, Theorem 1.3] it follows that $T$ is self-adjoint. In particular, if $A$ satisfies some conditions of self-adjointness, for instance those quoted in [13] then by \eqref{ec42} $A$ and $|{\rm Re}T|$ commute, therefore by the previous remark one has $|T|^2 \le |{\rm Re}T|^2$. But, in general the condition \eqref{ec11} does not impose other restriction on the contraction $A$ in \eqref{ec42}. In conclusion we doubt that \eqref{ec11} for $T$ implies always $T=T^*$.

Notice finally that an interesting context where it is possible to show that Fong-Tsui conjecture is true is that of $(A,m)$-expansive operators which were recently studied in [4] and [12], but we do not refer to them here.

Another interesting context for investigations on this conjecture is that of $A$-contractions (for some positive operator $A$ on $\h$), or even for $A$-bounded operators, which were extensively studied in the last years. Here is natural to use the concept of $A$-adjoint operator and $A$-projection, and an important role have the $A$-partial isometries recently investigated in [5, 6], [7], [9], or just the quasi-isometries in the context of $A$-contractions which were studied in [16].

\end{remark}

%{\bf Acknowledgements.}

%The second named author was supported by a Project financed from Lucian Blaga University of Sibiu research grants LBUS-IRG-2015-01.

\bigskip

\section*{References}
\medskip

[1] J. Agler and M. Stankus, {\it $m$-Isometric Transformations of Hilbert spaces}, Integr. Equat. Oper. Theory, 21, 4 (1995), 383-429.

[2] J. Agler and M. Stankus, {\it $m$-Isometric Transformations of Hilbert spaces II}, Integr. Equat. Oper. Theory, 23, 10 (1995), 1-48.

[3] J. Agler and M. Stankus, {\it $m$-Isometric Transformations of Hilbert spaces III}, Integr. Equat. Oper. Theory, 24 (1996), 379-421.

[4] O. A. M. Sid Ahmed and Adel Saddi, {\it $A-m$-Isometric operators in semi-Hilbertian spaces}, Linear Algebra Appl.
Vol. 436, Issue 10, Pages 3930–3942.

[5] M.L. Arias, G. Corach, and M.C. Gonzalez, Partial isometries in semi-Hilbertian spaces,
Linear Algebra Appl. 428 (2008), pp. 1460–1475.

[6] M. L. Arias, G Corach, and M.C. Gonzalez {\it Lifting properties in operator ranges}, Acta Sci.
Math. (Szeged) 75 (2009), pp. 635–653.

[7] M. L. Arias and M. Mbekhta, {\it $A$-partial isometries and generalized inverses}, Linear Algebra Appl.
Vol. 439, Issue 5, Pages 1286–1293

[8] G. Cassier and L. Suciu, {\it Mapping theorems and similarity to contractions for classes
of $A$-contractions}, Theta Series in Advanced Mathematics, (2008), 39-58.

[9] G. Corach, G. Fongi, and A. Maestripieri, {\it Weighted projections into closed subspaces}, Studia Math. 216 (2013), 131-148.

[10] C. K. Fong and V. Istr\u{a}\c tescu, {\it Some characterizations of hermitian operators and related classes of operators. I}, Proc. Amer. Math. Soc., 76, 1 (1979), 107-112.

[11] C. K. Fong and S. K. Tsui, {\it A note on positive operators}, J. Operator Theory, 5 (1981), 73-76.

[12] S. Jung, Y. Kim, E. Ko, and J. E. Lee, {\it On $(A,m)$-expansive operators}, Studia Mathematica, 213 (1) (2012), 3-23.

[13] I, H. Jeon, I. H. Kim, I. K. Tanahashi, and A. Uchiyama, {\it Conditions Implying Self-adjointness of Operators}, Integr. Equ. Oper. Theory, 61 (2008), 549-557.

[14] C. S. Kubrusly, {\it An introduction to Models and Decompositions in Operator Theory}. Birkh\"auser, Boston, 1997.

[15] M. Mbekhta and L. Suciu, {\it Classes of operators similar to partial isometries}, Integr. Equ. Oper. Theory, Vol. 63, No. 4, (2009), 571-590.

[16] M. Mbekhta and L. Suciu, {\it Quasi-isometries associated to $A$-contractions}, Linear Algebra and its Applications, Vol. 459, (2014), 430-453.

[17] M. H. Mortad, {\it A contribution to the Fong-Tsui conjecture related to self-adjoint operators}, arXiv:1208.4346v1 [math.FA], 21 Aug. 2012, 1-4.

[18] S. M. Patel, {\it A note on quasi-isometries}. Glasnik
    Matematicki, 35 (55) (2000), 307-312.

[19] S. M. Patel, {\it A note on quasi-isometries II}. Glasnik
    Matematicki, 38 (58) (2003), 111-120.

[20] L. Suciu, {\it Maximum subspaces related to $A$-contractions and quasinormal operators}, Journal of the Korean Mathematical Society, 45 (2008), No.1, 205-219.

[21] L. Suciu, {\it Maximum $A$-isometric part of an $A$-contraction and
applications}, Israel Journal of Mathematics, vol. 174, (2009), 419-442.

\end{document}